\documentclass[11pt]{amsart}
\usepackage{amsmath,amssymb}
\usepackage[dvips]{graphicx}
\usepackage{psfrag}
\title{Hamiltonian degree sequences in digraphs}
\author{Daniela K\"uhn, Deryk Osthus and Andrew Treglown}
\thanks {D.~K\"uhn and A.~Treglown were supported by the EPSRC, grant no.~EP/F008406/1.
D.~Osthus was supported by the EPSRC, grant no.~EP/E02162X/1 and~EP/F008406/1.}
\date{} 
\def\COMMENT#1{}
\def\TASK#1{}
\begin{document}
\maketitle
\def\noproof{{\unskip\nobreak\hfill\penalty50\hskip2em\hbox{}\nobreak\hfill%
        $\square$\parfillskip=0pt\finalhyphendemerits=0\par}\goodbreak}
\def\endproof{\noproof\bigskip}
\newdimen\margin   
\def\textno#1&#2\par{%
    \margin=\hsize
    \advance\margin by -4\parindent
           \setbox1=\hbox{\sl#1}%
    \ifdim\wd1 < \margin
       $$\box1\eqno#2$$%
    \else
       \bigbreak
       \hbox to \hsize{\indent$\vcenter{\advance\hsize by -3\parindent
       \sl\noindent#1}\hfil#2$}%
       \bigbreak
    \fi}
\def\proof{\removelastskip\penalty55\medskip\noindent{\bf Proof. }}
\def\C{\mathcal{C}}
\def\eps{\varepsilon}
\def\ex{\mathbb{E}}
\def\prob{\mathbb{P}}
\def\eul{{\rm e}}
\newtheorem{firstthm}{Proposition}
\newtheorem{thm}[firstthm]{Theorem}
\newtheorem{prop}[firstthm]{Proposition}
\newtheorem{lemma}[firstthm]{Lemma}
\newtheorem{cor}[firstthm]{Corollary}
\newtheorem{problem}[firstthm]{Problem}
\newtheorem{defin}[firstthm]{Definition}
\newtheorem{conj}[firstthm]{Conjecture}
\newtheorem{claim}[firstthm]{Claim}
\begin{abstract}
We show that for each $\eta>0$ every digraph~$G$ of sufficiently large order~$n$
is Hamiltonian if its out- and indegree sequences $d^+ _1\le \dots \le d^+ _n$
and $d^- _1 \le \dots \le d^- _n$ satisfy (i)
$ d^+ _i \geq i+ \eta n  $ or $ d^- _{n-i- \eta n} \geq n-i $ and
(ii) $ d^- _i \geq i+ \eta n $ or $ d^+ _{n-i- \eta n} \geq n-i $ for all $i < n/2$.
This gives an approximate solution to a problem of Nash-Williams~\cite{nw}
concerning a digraph analogue of Chv\'atal's theorem.
In fact, we prove the stronger result that such digraphs~$G$ are pancyclic.
\end{abstract}
\section{Introduction}\label{sec1}
Since it is unlikely that there is a characterization of all those graphs which
contain a Hamilton cycle it is natural to ask for sufficient conditions
which ensure Hamiltonicity. One of the most general of these is Chv\'atal's
theorem~\cite{ch} that characterizes all those degree sequences which
ensure the existence of a Hamilton cycle in a graph: Suppose that the degrees of the graph
are $d_1\le \dots \le d_n$. If $n \geq 3$ and $d_i \geq i+1$ or $d_{n-i} \geq n-i$
for all $i <n/2$ then $G$ is Hamiltonian. This condition on the degree sequence is
best possible in the sense that for any degree sequence violating this condition there
is a corresponding graph with no Hamilton cycle. More precisely, if $d_1 \leq \dots \leq d_n$
is a graphical degree sequence (i.e. there exists a graph with this degree sequence) then there
exists a non-Hamiltonian graph $G$ whose degree sequence $d'_1 \leq  \dots \leq d'_n$ is such
that $d'_i \geq d_i$ for all $1\leq i \leq n$.

A special case of Chv\'atal's theorem is Dirac's theorem, which states that every graph
with $n \geq 3$ vertices and minimum degree at least $n/2$ has a Hamilton cycle. An analogue
of Dirac's theorem for digraphs was proved by Ghouila-Houri~\cite{gh}. (The digraphs
we consider do not have loops and we allow at most one edge in each direction between
any pair of vertices.) Nash-Williams~\cite{nw}
raised the question of a digraph analogue of Chv\'atal's theorem quite soon after the
latter was proved. 

For a digraph~$G$ it is natural to consider both its outdegree sequence $d^+ _1,\dots , d^+ _n$
and its indegree sequence $d^- _1,\dots , d^- _n$. Throughout this paper we take the convention that
$d^+ _1\le \dots \le d^+ _n$ and $d^- _1 \le \dots \le  d^- _n$ without mentioning this explicitly.
Note that the terms $d^+ _i$ and $d^- _i$ do not necessarily correspond to the degree of the same vertex
of~$G$. 
\begin{conj}[Nash-Williams~\cite{nw}]\label{nw}
Suppose that $G$ is a strongly connected digraph on $n \geq 3$ vertices
such that for all $i < n/2$
\begin{itemize}
\item[(i)]  $d^+ _i \geq i+1 $ or $ d^- _{n-i} \geq n-i $,
\item[(ii)] $ d^- _i \geq i+1$ or $ d^+ _{n-i} \geq n-i.$
\end{itemize}
Then $G$ contains a Hamilton cycle.
\end{conj}
No progress has been made on this conjecture so far (see also~\cite{bang}).
It is even an open problem whether the conditions imply the existence of a cycle through
any pair of given vertices (see~\cite{bt}).

As discussed in Section~\ref{extremal}, one cannot omit the condition that~$G$ is strongly connected. 
At first sight one might also try to replace the degree condition in Chv\'atal's theorem by 
\begin{itemize}
\item $d^+ _i \geq i+1 $ or $d^+ _{n-i} \geq n-i$, 
\item $d^- _i \geq i+1 $ or $d^- _{n-i} \geq n-i$. 
\end{itemize} 
However, Bermond and Thomassen~\cite{bt} observed that the latter conditions do not
guarantee Hamiltonicity. Indeed, consider the digraph obtained from the complete digraph~$K$
on $n-2\ge 4$ vertices by adding two new vertices~$v$ and~$w$ which both send an edge to
every vertex in~$K$ and receive an edge from one fixed vertex $u\in K$.%
    \COMMENT{For all $x \in K$, $d^- (x) =n-1$. So since $|K|> n/2$ certainly $d^- _{n-i} \geq n-i$ for all $i <n/2$. 
Furthermore, $d^+ (u)=n-1$, $d^+ (v)=d^+ (w)=n-2$ and $d^+ (x)=n-3 \ge \lceil n/2 \rceil -1+1$ for all $x \in K$. So $d^+ _{i} \geq i+1$
for all $1 \le i< n/2$.}

The following example shows that the 
degree condition in Conjecture~\ref{nw} would be best possible in the sense that for all
$n\ge 3$ and all $k<n/2$ there is a non-Hamiltonian strongly connected digraph~$G$ on~$n$ vertices
which satisfies the degree condition except that $d^+_k,d^-_k\ge k+1$ are replaced by
$d^+_k,d^-_k\ge k$ in the $k$th pair of conditions.
To see this, take an independent set~$I$ of size $k<n/2$ 
and a complete digraph~$K$ of order~$n-k$. Pick a set~$X$ of~$k$ vertices of~$K$
and add all possible edges (in both directions) between~$I$ and~$X$. The digraph~$G$
thus obtained is strongly connected, not Hamiltonian and
$$\underbrace{k, \dots ,k}_{k \text{ times}}, \underbrace{n-1-k, \dots , n-1-k}_{n-2k
\text{ times}}, \underbrace{n-1, \dots , n-1}_{k \text{ times}}$$ is both the out- and
indegree sequence of~$G$.
A more detailed discussion of extremal examples is given in Section~\ref{extremal}.
    \COMMENT{This comment is now obsolete?
We know of the existence of non-Hamiltonian digraphs $G$ that satisfy the degree conditions in
Conjecture~\ref{nw}, except $d^+ _{k_1} \geq k_1+1$ and $d^+ _{k_2}\geq k_2 +1$ are replaced by 
$d^+ _{k_1} \geq k_1$ and $d^+ _{k_2}\geq k_2 $ respectively, for $k_1 \not = k_2 <|G|/2$.
Indeed, let $n \geq 6$ be even and $K$ be a complete digraph on $n/2+1$ vertices where
$V(K)=\{u,v,x_1,\dots,x_{n/2-1}\}$. Let $I$ be an independent set of $n/2-1$ vertices such that
$I=\{y_1, \dots, y_{n/2-1}\}$. We obtain the digraph $G$ from the disjoint union of $K$ and $I$
by adding edges $x_iy_i, y_ix_i \in E(G)$ precisely when $i \leq j$, and adding all possible edges between
$u$ and $I$ (in both directions).
$G$ has in- and outdegree sequence $2,3, \dots, n/2-1, n/2 , n/2 ,n/2+1 , \dots , n-2 , n-1 ,n-1$ and
contains precisely two Hamilton cycles, namely $uy_1 x_1 \dots y_{n/2-1} x_{n/2-1} v $
and $v x_{n/2-1} y_{n/2-1} \dots x_1 y_1 u$.
Given $1\leq k_1 <k_2\leq n/2-1$ remove the edges $y_{k_1} x_{k_1}$ and $y_{k_2}x_{k_2-1}$ from $G$. 
So $G$ in now non-Hamiltonian. However, since no vertex in $I$ lost an inneighbour, the first 
$n/2-1$ terms in the indegree sequence of $G$ remain unchanged. Of the first $n/2-1$ terms
in the outdegree sequence of $G$ only the $(k_1)$th and $(k_2)$th terms have been altered.
}

In this paper we prove the following approximate version of Conjecture~\ref{nw} for large digraphs. 
\begin{thm}\label{approxnw}
For every $\eta >0$ there exists an integer $n_0 =n_0 (\eta)$ such that the following holds.
Suppose $G$ is a digraph on $n \geq n_0$ vertices  such that for all $i < n/2$
\begin{itemize}
\item $ d^+ _i \geq i+ \eta n  $ or $ d^- _{n-i- \eta n} \geq n-i $,
\item $ d^- _i \geq i+ \eta n $ or $ d^+ _{n-i- \eta n} \geq n-i .$
\end{itemize}Then $G$ contains a Hamilton cycle.
\end{thm}
Instead of proving Theorem~\ref{approxnw} directly, we will prove the existence of a
Hamilton cycle in a digraph satisfying a certain expansion property (Theorem~\ref{expanderthm}).
We defer the precise statement to Section~\ref{sec4}.

The following weakening of Conjecture~\ref{nw} was posed earlier by Nash-Williams~\cite{ch1,ch2}.
It would yield a digraph analogue of P\'osa's theorem which states that a graph~$G$ on~$n\ge 3$
vertices has a Hamilton cycle if its degree sequence $d_1, \dots , d_n$ satisfies $d_i \geq i+1$ 
for all $i<(n-1)/2$ and if additionally $d_{\lceil n/2\rceil} \geq \lceil n/2\rceil$ when~$n$ is odd~\cite{posa}. 
Note that this is much stronger than Dirac's theorem but is a special case of Chv\'atal's theorem.
\begin{conj}[Nash-Williams~\cite{ch1,ch2}]\label{nw2}
Let $G$ be a digraph on $n \geq 3$ vertices such that $d^+ _i,d^-_i \geq i+1 $
for all $i <(n-1)/2$ and such that additionally
$d^+_{\lceil n/2\rceil},d^-_{\lceil n/2\rceil} \geq \lceil n/2\rceil$ when~$n$ is odd. 
Then~$G$ contains a Hamilton cycle.
\end{conj}
The previous example shows that the degree condition would be best possible in the same sense as described there. 
The assumption of strong connectivity is not necessary in Conjecture~\ref{nw2},
as it follows from the degree conditions.
The following approximate version of Conjecture~\ref{nw2}
is an immediate consequence of Theorem~\ref{approxnw}.
\begin{cor}\label{posa}
For every $\eta >0$ there exists an integer $n_0 =n_0 (\eta)$ such that every digraph~$G$
on $n \geq n_0$ vertices with $ d^+ _i, d^-_i \geq i+ \eta n  $ for all $i < n/2$
contains a Hamilton cycle.
\end{cor}
In Section~\ref{orient} we give a construction which shows that for oriented graphs
there is no analogue of P\'osa's theorem. (An oriented graph is a digraph with no $2$-cycles.)

It will turn out that the conditions of Theorem~\ref{approxnw} even
guarantee the digraph~$G$ to be \emph{pancyclic}, i.e.~$G$ contains a cycle of length~$t$
for all $t=2,\dots,n$.
\begin{cor}\label{pancyclic}
For every $\eta >0$ there exists an integer $n_0 =n_0 (\eta)$ such that the following holds.
Suppose $G$ is a digraph on $n \geq n_0$ vertices such that for all $i < n/2$
\begin{itemize}
\item $ d^+ _i \geq i+ \eta n $ or $ d^- _{n-i- \eta n} \geq n-i $,
\item $ d^- _i \geq i+ \eta n \text{ or } d^+ _{n-i- \eta n} \geq n-i .$
\end{itemize} Then $G$ is pancyclic.
\end{cor}
Thomassen~\cite{tom} proved an Ore-type condition which
implies that  every digraph with  minimum
in- and outdegree $>n/2$ is pancyclic. (The complete bipartite digraph whose vertex class sizes
are as equal as possible shows that the latter bound is best possible.)
Alon and Gutin~\cite{ag} observed that one can use 
Ghouila-Houri's theorem to show that every digraph $G$ with minimum
in- and outdegree $>n/2$ is even vertex-pancyclic.%
   \COMMENT{In the digraphs book it says degree at least $n/2+1$. But $>n/2$ already works
and is best possible (complete bip digraph).}
Here a digraph~$G$ is called {\it vertex-pancyclic} if every vertex of~$G$ lies
on a cycle of  length~$t$ for all $t=2,\dots,n$. In Proposition~\ref{vertexpan} we show that
one cannot replace pancyclicity by vertex-pancyclicity in Corollary~\ref{pancyclic}.
Minimum degree conditions for (vertex-) pancyclicity of oriented graphs are discussed
in~\cite{kelly2}. 

Our result on Hamilton cycles in expanding digraphs (Theorem~\ref{expanderthm}) is used as a
tool in~\cite{KMO} to prove an approximate version of Sumner's universal tournament conjecture.
Theorem~\ref{expanderthm} also has an application
to a conjecture of Thomassen on tournaments.
A \emph{tournament} is an orientation of a complete graph. We say that a tournament is {\it regular} if every vertex has 
equal in- and outdegree. Thus
regular tournaments contain an odd number $n$ of vertices and each vertex has in- and outdegree  $(n-1)/2$. 
It is easy to see that every regular tournament contains a Hamilton cycle. Thomassen~\cite{tom1}
conjectured that even if we remove a number of
edges from a regular tournament $G$, the remaining oriented graph still contains a Hamilton cycle.
\begin{conj}[Thomassen~\cite{tom1}] \label{thomconj}
If $G$ is a regular tournament on $n$ vertices and $A$ is any set of less than
$(n-1)/2$ edges of $G$, then $G-A$ contains a Hamilton cycle.
\end{conj}
In Section~\ref{torn} we prove Conjecture~\ref{thomconj} for sufficiently large regular tournaments.
Note that Conjecture~\ref{thomconj} is a weakening of the following conjecture of Kelly~(see e.g.~\cite{bang,bondy,moon}).
\begin{conj}[Kelly] \label{kelly}
Every regular tournament on $n$ vertices can be decomposed
into $(n-1)/2$ edge-disjoint Hamilton cycles.
\end{conj}
In~\cite{kelly} we showed that every sufficiently large regular tournament can be `almost' decomposed into edge-disjoint
Hamilton cycles, thus giving an approximate solution to Kelly's conjecture.

This paper is organized as follows. We first give a more detailed discussion of extremal 
examples for Conjecture~\ref{nw}. After introducing some basic notation, in Section~\ref{sec2}
we then deduce Corollary~\ref{pancyclic} from Theorem~\ref{approxnw} and show that one cannot replace
pancyclicity by vertex-pancyclicity. Our proof of Theorem~\ref{approxnw} uses the
Regularity lemma for digraphs which, along with other tools, is introduced in Section~\ref{sec3}.
The proof of Theorem~\ref{approxnw} is included in Section~\ref{sec4}. It
relies on a result (Lemma~\ref{cyclelemma}) from joint work~\cite{kko} of the first two authors
with Keevash on an analogue of Dirac's theorem for oriented graphs. A related result was
proved earlier in~\cite{approxham}.

It is a natural question to ask whether the `error terms' in Theorem~\ref{approxnw} and
Corollary~\ref{posa} can be eliminated using an `extremal case' or `stability' analysis.
However, this seems quite difficult as there are many different types of digraphs which 
come close to violating the conditions in Conjectures~\ref{nw} and~\ref{nw2}
(this is different e.g.~to the situation in~\cite{kko}). As a step in this direction, very recently it was
shown in~\cite{CKKOsemi} that the degrees in the first parts of the conditions in
Theorem~\ref{approxnw} can be capped at $n/2$, i.e.~the conditions can be replaced by
\begin{itemize}
\item $ d^+ _i \geq \min\{i+ \eta n,n/2\}  $ or $ d^- _{n-i- \eta n} \geq n-i $,
\item $ d^- _i \geq \min\{i+ \eta n,n/2\} $ or $ d^+ _{n-i- \eta n} \geq n-i .$
\end{itemize}
The proof of this result is considerably more difficult than that of Theorem~\ref{approxnw}.
A (parallel) algorithmic version of Chv\'atal's theorem for undirected graphs was recently considered
in~\cite{Sarkozy} and for directed graphs in~\cite{CKKOalgo}. 

\section{Extremal examples for Conjecture~\ref{nw} and a weaker conjecture} \label{extremal}

The example given in the introduction does not quite imply that Conjecture~\ref{nw}
would be best possible, as for some~$k$ it violates both~(i) and~(ii) for $i=k$.
Here is a slightly more complicated example which only violates one of the conditions
for $i=k$ (unless~$n$ is odd and $k=\lfloor n/2 \rfloor$).

Suppose%
    \COMMENT{If $n=4$ then the graph obtained by reversing the edges works.}
$n\geq 5$ and $1\leq k < n/2$.
Let~$K$ and~$K'$ be complete digraphs on $k-1$ and
$n-k-2$ vertices respectively. Let $G$ be the digraph on $n$ vertices obtained from the
disjoint union of~$K$ and~$K'$ as follows. Add all possible edges
from~$K'$ to~$K$ (but no edges from~$K$ to~$K'$) and add
new vertices~$u$ and~$v$ to the digraph such that there are all possible edges 
from~$K'$ to~$u$ and~$v$ and all possible edges from~$u$ and~$v$ to~$K$.
Finally, add a vertex $w$ that sends and receives edges from all other vertices of~$G$
(see Figure~1).
\begin{figure}[htb!]
\begin{center}\footnotesize
\psfrag{1}[][]{$K'$}
\psfrag{2}[][]{$K$}
\psfrag{w}[][]{$w$}
\psfrag{v}[][]{$v$}
\psfrag{u}[][]{$u$}
\includegraphics[width=0.5\columnwidth]{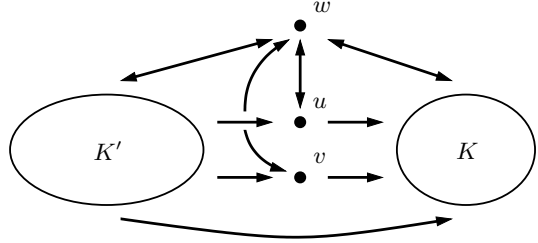}  
\caption{An extremal example for Conjecture~\ref{nw}}
\label{fig:bridges}
\end{center}
\end{figure}
Thus $G$ is strongly connected, not Hamiltonian and has outdegree sequence
\begin{center}$\underbrace{k-1, \dots ,k-1}_{k-1 \text{ times}}, k, k,
\underbrace{n-1, \dots , n-1}_{n-k-1 \text{ times}}$
\end{center}
and indegree sequence
\begin{center}$\underbrace{n-k-2, \dots ,n-k-2}_{n-k-2 \text{ times}},
n-k-1, n-k-1, \underbrace{n-1, \dots , n-1}_{k \text{ times}}.$
\end{center}
Suppose that either~$n$ is even or, if~$n$ is odd, we have that $k< \lfloor n/2 \rfloor$.
One can check that~$G$ then satisfies the conditions in Conjecture~\ref{nw} except that
$d^+ _{k} = k$ and $d^-_{n-k} =n-k-1$. (When checking the conditions, it is convenient to note
that our assumptions on $k$ and $n$ imply $n-k-1 \ge  \lceil n/2 \rceil$. Hence there are at least
$\lceil n/2 \rceil$ vertices of outdegree $n-1$ and so~(ii) holds for all $i<n/2$.)%
         \COMMENT{Suppose that~$n$ is even. Note that in this case we have
$n-k-1 \ge n-(n/2-1)-1=n/2$. Thus there are $n/2$ vertices of outdegree $n-1$
and so (ii) holds for all $i <n/2$. Clearly~(i) is satisfied for all $i \le k-2$ and all
$k+2 \le i < n/2$.
Now note that 
$d^-_{n-(k-1)}=n-1$ so~(i) holds for $i=k-1$. 
Also $d^-_{n-(k+1)}=n-k-1$ so~(i) holds for $i=k+1$.
Now suppose that $n$ is odd and $ k \le \lfloor  n/2\rfloor-1$.
Then $n-k-1 \ge n- (\lfloor n/2 \rfloor -1)-1 =\lceil n/2 \rceil$. 
So (ii) holds for all $i < n/2$.  The argument for (i) is as before.
Now let $k=\lfloor n/2 \rfloor$. Then in the above example we have
$d^+_k=k$ as before but $d^-_{n-k}=n-k-1$ so~(i) fails.
Also we have $n-k-2=\lceil n/2 \rceil-2=k-1$ and so 
$d^-_{k}=n-k-1=\lceil n/2 \rceil -1=k$ and also
$d^+_{n-k}=d^+_{\lceil n/2 \rceil} =d^+_{k+1}=k<n-k$. So~(ii) fails too.}
If~$n$ is odd and $k=\lfloor n/2 \rfloor$ then conditions~(i) and~(ii) both fail for $i=k$.
We do not know whether a similar construction as above also exists for this case.
It would also be interesting to find an analogous construction as above for Conjecture~\ref{nw2}.

Here is also an example which shows that the assumption of strong connectivity in 
Conjecture~\ref{nw} cannot be omitted. Let $n \geq 4$ be even. Let~$K$ and~$K'$ be 
two disjoint copies of a complete digraph on~$n/2$ vertices.
Obtain a digraph~$G$ from~$K$ and~$K'$ by adding all possible edges from~$K$ to~$K'$ 
(but none from~$K'$ to~$K$). It is easy to see that~$G$ is neither Hamiltonian,
nor strongly connected, but satisfies the condition on the 
degree sequences given in Conjecture~\ref{nw}.

As it stands, the additional connectivity assumption means that 
Conjecture~\ref{nw} does not seem to be a precise digraph analogue of Chv\'atal's theorem:  
in such an analogue, we would ask for a complete characterization of \emph{all}
digraph degree sequences which force Hamiltonicity.
However, it turns out that it makes sense to  replace the strong connectivity assumption with an additional 
degree condition (condition~(iii) below). If true, the following conjecture would provide
the desired characterization.
\begin{conj}\label{nw3}
Suppose that $G$ is a digraph on $n \geq 3$ vertices
such that for all $i < n/2$
\begin{itemize}
\item[(i)]  $d^+ _i \geq i+1 $ or $ d^- _{n-i} \geq n-i $,
\item[(ii)] $ d^- _i \geq i+1$ or $ d^+ _{n-i} \geq n-i$,
\end{itemize}
and such that~{\rm (iii)} $d^+ _{n/2} \geq n/2 $ or $ d^- _{n/2} \geq n/2$ if~$n$ is even.
Then $G$ contains a Hamilton cycle.
\end{conj}
Conjecture~\ref{nw3} would actually follow from Conjecture~\ref{nw}. To see this, it of course 
suffices to check that the conditions in Conjecture~\ref{nw3} imply strong connectivity.
This in turn is easy to verify, as the degree conditions imply that for any vertex set~$S$ 
with $|S| \le n/2$ we have $|N^-(S) \cup S|>|S|$ and $|N^+(S) \cup S| > |S|$. 
(We need~(iii) to obtain this assertion precisely for those
$S$ with $|S|=n/2$.)%
     \COMMENT{For this, consider any set $S$ with $|S| < n/2$. If $d^+_{|S|} \ge |S|+1$,
we have a vertex of outdegree at least $|S|+1$ in $S$ and so $|N^+(S)| \ge |S|+1$.
Otherwise, we have $d^-_{n-|S|} \ge n-|S|$ and so we have a vertex of indegree
$n-|S|$ outside $S$ and so $|N^+(S)| \ge |S|+1$ again. 
Now suppose $|S| = n/2$ (so $n$ is even). If $d^+ _{n/2} \geq n/2 $, then we have a vertex $x$ 
of outdegree at least $n/2$ in~$S$. So~$x$ has an outneighbour outside $S$. 
If $ d^- _{n/2} \geq n/2$, then $\bar{S}$ has a vertex $x$ with an inneighbour from $S$.
Altogether this shows that for any vertex $y$, the set $Y^+$ of vertices $z$ for which
there is an $y$-$z$ path satisfies $|Y^+| > n/2$. Similarly, we have $|Y^-| > n/2$.
This implies strong connectivity.}

It remains to check that Conjecture~\ref{nw3} would indeed characterize all 
digraph degree sequences which force a Hamilton cycle. Unless~$n$ is odd and $k=\lfloor n/2 \rfloor$,
the construction at the beginning of the section already gives non-Hamiltonian graphs 
which satisfy all the degree conditions (including (iii))%
     \COMMENT{To check~(iii), suppose that~$n$ is even. If $n-k-1\ge n/2+1$, ie if $k\le n/2-2$ then
$d^+_{n/2}=n-1$, ie (iii) holds. If $k=n/2-1$ then $n-k-2=n/2-1$ and so $d^-_{n/2}=n-k-1=n/2$,
ie (iii) holds.} 
except~(i) for $i=k$. To cover the case when~$n$ is odd and $k=\lfloor n/2 \rfloor$,
let~$G$ be the digraph obtained from two disjoint cliques~$K$ and~$K'$ of orders $\lceil n/2 \rceil$
and $\lfloor n/2 \rfloor$ by adding all edges from~$K$ to~$K'$. 
If $i=k=\lfloor n/2 \rfloor$ then~$G$ satisfies (ii) (because $d^+_{n-k}=n-1$) but not~(i).
For all other~$i$, both conditions are satisfied.%
     \COMMENT{Let $k=\lfloor n/2 \rfloor$. We have~(i) $d^+_k=k-1$ and
$d^-_{n-k}=\lceil n/2 \rceil-1=n-k-1$. Also (ii) $d^-_k=k$ and $d^+_{n-k}=n-1$.
For $k'=\lfloor n/2 \rfloor-1$ both conditions are satisfied:
(i) $d^+_{k'}=k'$ and $d^-_{n-k'}=n-1$. Also (ii) $d^-_{k'}=k'+1$.}
Finally, the example immediately preceding Conjecture~\ref{nw3} gives a graph on
an even number~$n$ of vertices which satisfies~(i) and~(ii) for all $i < n/2$ but does
not satisfy~(iii).%
     \COMMENT{Indeed, it has $d^+_{n/2}= d^-_{n/2}=n/2-1$.} 

Nash-Williams observed that Conjecture~\ref{nw} would imply Chv\'atal's theorem.
(Indeed, given an undirected graph~$G$ satisfying the degree condition in Chv\'atal's theorem,
obtain a digraph by replacing each undirected edge with a pair of directed edges, one
in each direction. This satisfies the degree condition in Conjecture~\ref{nw}.
It is also strongly connected, as it is easy to see that~$G$ must be connected.)
A disadvantage of Conjecture~\ref{nw3} is that it would not imply Chv\'atal's theorem in the same way:
consider a graph~$G$ which is obtained from $K_{n/2,n/2}$ by removing a perfect matching and
adding a spanning cycle in one of the two vertex classes. The degree sequence of this $G$
satisfies the conditions of Chv\'atal's theorem. However, the digraph obtained by doubling the edges of~$G$
does not satisfy~(iii) in Conjecture~\ref{nw3}.
 
\section{Notation and the proof of Corollary~\ref{pancyclic}}\label{sec2}
We begin this section with some notation. 
Given two vertices~$x$ and~$y$ of a digraph~$G$, we write~$xy$ for the edge directed from~$x$ to~$y$.
The order~$|G|$ of~$G$ is the number of its vertices.
We denote by~$N^+ _G (x)$ and~$N^- _G (x)$ the out- and the inneighbourhood of~$x$
and by $d^+_G(x)$ and $d^-_G(x)$ its out- and indegree.
We will write $N^+ (x)$ for example, if this is unambiguous. Given $S \subseteq V(G)$,
we write~$N^+ _G (S)$ for the union of~$N^+ _G (x)$ for all $x \in S$
and define~$N^- _G(S)$ analogously. The {\it minimum semi-degree} $\delta ^0 (G)$ of~$G$
is the minimum of its minimum outdegree~$\delta ^+ (G)$ and its minimum indegree~$\delta ^- (G)$.


{\removelastskip\penalty55\medskip\noindent{\bf Proof of Corollary~\ref{pancyclic}. }}
Our first aim is to prove the existence of
a vertex $x \in V(G)$ such that $d^+ (x)+d^- (x) \geq n$. Such a vertex exists if there is
an index~$j$ with $d^+ _j  + d^- _{n-j} \geq n$. Indeed, at least~$n-j+1$ vertices of~$G$
have outdegree at least~$d^+_j$ and at least $j+1$ vertices have indegree at least~$d^- _{n-j}$.
Thus there will be a vertex~$x$ with $d^+(x)\ge d^+_j$ and $d^-(x)\ge d^- _{n-j}$.

To prove the existence of such an index~$j$, suppose first that there is an~$i$
with $2\le i<n/2$ and such that
$d^+_{i-1}\ge i$ but $d^+_i= i$. Then $d^-_{n-i}\ge n-i$ and so
$d^+_i+d^-_{n-i}\ge n$ as required. The same argument works if there
is an~$i$ with $2\le i<n/2$ and such that $d^-_{i-1}\ge i$ but $d^-_i= i$.
Suppose next that $d^+_1\le 1$. Then $d^-_{n-1}\ge n-1$ and so $d^+_1=1$. Thus we can take $j:=1$.
Again, the same argument works if $d^-_1\le 1$. Thus we may assume that
$d^+_{\lceil{n/2}\rceil-1}, d^-_{\lceil{n/2}\rceil-1} \ge \lceil{n/2}\rceil$.
But in this case we can take $j:=\lfloor n/2\rfloor$.

Now let~$x$ be a vertex with $d^+ (x)+d^- (x) \geq n$, set $G':=G-x$ and
$n':=|G'|$. Let $d^+_{1,G'},\dots, d^+_{n',G'}$ and $d^-_{1,G'},\dots, d^-_{n',G'}$
denote the out- and the indegree sequences of~$G'$. Given some $i \le n'$ and $s>0$, if
$d^+ _i \geq s$ then at least $n+1-i$ vertices in~$G$ have outdegree at least~$s$.
Thus at least~$n-i=n'+1-i$ vertices in~$G'$ have outdegree at least~$s-1$ and so
$d^+ _{i,G'} \geq s-1$. Thus for all $i < n/2$ the degree sequences of~$G'$
satisfy 
\begin{itemize}
\item $ d^+ _{i,G'} \geq i+ \eta n -1 $ or $ d^- _{n-i- \eta n, G'} \geq n-i-1 $,
\item $ d^- _{i,G'} \geq i+ \eta n -1 $ or $ d^+ _{n-i- \eta n,G'} \geq n-i-1 $
\end{itemize} 
and so
\begin{itemize}
\item $ d^+ _{i,G'} \geq i+ \eta n'/2 $ or $ d^- _{n'-i- \eta n'/2,G'} \geq n'-i $,
\item $ d^- _{i,G'} \geq i+ \eta n'/2 $ or $ d^+ _{n'-i- \eta n'/2,G'} \geq n'-i.$
\end{itemize} 
Hence we can apply Theorem~\ref{approxnw} with $\eta$ replaced by $\eta/2$ to obtain a 
Hamilton cycle~$C=x_1\dots x_{n'}$ in $G'$.
We now apply the same trick as in~\cite{ag} to obtain a cycle (through~$x$) in $G$ of the desired
length, $t$ say (where $2\le t\le n$): Since $d^+_G (x)+d^-_G (x) \ge n > n'$ there 
exists an $i$ such that $x_i \in N^+ _G (x)$ and $x_{i+t-2} \in N^- _G (x)$
(where we take the indices modulo~$n'$). But then $xx_i\dots x_{i+t-2}x$ is the
required cycle of length~$t$.
\endproof
Note that the proof of Corollary~\ref{pancyclic} shows that if
Conjecture~\ref{nw} holds and~$G$ is a strongly $2$-connected digraph with
\begin{itemize}
\item $d^{+}_i \ge i+2$ or  $d^{-} _{n-i-1} \geq n-i$,
\item $d^{-}_i \ge i+2$ or  $d^{+} _{n-i-1} \geq n-i$
\end{itemize}
for all $i<n/2$ then~$G$ is pancyclic. %
\COMMENT{ A complete bipartite digraph with vertex classes of size
$n/2$ has $d^+_{n/2-1}=n/2<(n/2-1)+2$ and $d^-_{n-(n/2-1)-1}=d^-_{n/2}=n/2<n-(n/2-1)$.
So it just violates the condition for $i=n/2-1$ and doesn't contain odd cycles.
Is the condition best possible for each~$i$?}

The next result implies that we cannot replace pancyclicity with vertex-pancyclicity
in Corollary~\ref{pancyclic}.
\begin{prop}\label{vertexpan}
Given any $k \geq 3$ there are $\eta=\eta(k)>0$ and $n_0=n_0(k)$ such that for every
$n \geq n_0$ there exists a digraph~$G$ on~$n$ vertices with $d^{+} _i,d^-_i \geq i+ \eta n $
for all $i<n/2$, but such that some vertex of~$G$
does not lie on a cycle of length less than~$k$.
\end{prop}
\proof
Let $\eta :=1/(k3^k)$ and suppose that~$n$ is sufficiently large. Let~$G$ be the digraph obtained
from the disjoint union of~$k-2$ independent sets $V_1, \dots , V_{k-2}$
with $|V_i|=3^i\lceil \eta n\rceil$ and a complete digraph~$K$ on $n-1-|V_1\cup \dots\cup V_{k-2}|$ vertices
as follows. Add a new vertex~$x$ which sends an edge to all vertices in~$V_1$ and receives
an edge from all vertices in~$K$. Add all possible edges from~$V_i$ to~$V_{i+1}$
(but no edges from~$V_{i+1}$ to~$V_i$) for each $i \leq k-3$.
Finally, add all possible edges going from vertices in~$K$ to other vertices 
and add all edges from $V_{k-2}$ to $K$.
Then $d^- _i \geq |K|\ge 2n/3$ and
$d^{+} _i \geq i+ \eta n $ for all $i<n/2$ with room to spare.
However, if~$C$ is a cycle containing $x$ then the inneighbour of~$x$ on~$C$ must lie in~$K$.
But the shortest path from~$x$ to~$K$ has length~$k-1$ and so $|C|\geq k$, as required.
\endproof

\section{Degree sequences for Hamilton cycles in oriented graphs} \label{orient}

In Section~\ref{sec1} we mentioned Ghouila-Houri's  theorem which gives a bound on the
minimum semi-degree of a digraph~$G$ guaranteeing a Hamilton cycle. A natural question raised
by Thomassen~\cite{thomassen_79_long_cycles_constraints} is that of determining the minimum
semi-degree which ensures a Hamilton cycle
in an oriented graph. H\"aggkvist~\cite{HaggkvistHamilton} conjectured that every oriented graph~$G$ of order
$n\geq 3$ with $\delta ^0 (G) \geq (3n-4)/8$ contains a 
Hamilton cycle. The bound on the minimum semi-degree would be best possible.
The first two authors together with Keevash~\cite{kko} confirmed this conjecture
for sufficiently large oriented graphs. 

P\'osa's theorem implies the existence of a Hamilton cycle in a graph~$G$ even if~$G$
contains a significant number of vertices of degree much less than $n/2$, i.e.~of degree
much less than the minimum degree required to force a Hamilton cycle. In particular,
P\'osa's theorem is much stronger than Dirac's theorem. In the same sense, 
Conjecture~\ref{nw2} would be much stronger than Ghouila-Houri's theorem.
The following proposition implies that we cannot strengthen H\"aggkvist's conjecture in this way:
there are non-Hamiltonian oriented graphs which contain just a bounded number of vertices whose
semi-degree is (only slightly) smaller than $3n/8$.
To state this proposition we need to introduce the notion of dominating sequences: 
Given sequences $x_1,\dots,x_n$ and $y_1, \dots, y_n$ of numbers we say
that $y_1, \dots, y_n$ \emph{dominates} $x_1,\dots,x_n$  if $x_i \leq y_i$ for all $1\le i \leq n$.

\begin{prop}\label{orientedprop}
For every $0 < \alpha <3/8$, there is an integer $c=c(\alpha)$ and
infinitely many oriented graphs~$G$ whose in- and outdegree sequences both
dominate
$$\underbrace{\alpha |G|, \dots, \alpha|G|}_{c \text{ times}},3|G|/8, \dots , 3 |G|/8$$
but such that $G$ does not contain a Hamilton cycle.
\end{prop}
\begin{proof}
Define $c:=4t$ where $t \in \mathbb N$ is chosen such that $3-1/t>8 \alpha$.
Let $n$ be sufficiently large and such that~$8t$ divides~$n$ and define vertex sets
$A,B,C,D$ and $E$ of sizes $n/4, n/8,n/8-1,n/4+1$ and $n/4$ respectively.

Let~$G$ be the oriented graph obtained from the disjoint union of $A,B,C,D$ and~$E$
by defining the following edges: $G$ contains all possible edges from $A$ to~$B$, $B$ to~$C$, $C$ to~$D$,
$A$ to~$C$, $B$ to~$D$ and $D$ to~$A$. $E$ sends out all possible edges to~$A$ and~$B$ and receives
all possible edges from~$C$ and~$D$. $B$ and $C$ both induce tournaments that are as regular as possible
(see Figure~2).
\begin{figure}[htb!]
\begin{center}
\psfrag{1}[][]{$A$}
\psfrag{2}[][]{ $B$}
\psfrag{3}[][]{ $C$}
\psfrag{4}[][]{ $D$}
\psfrag{5}[][]{ $E$}
\psfrag{a}[][]{ $n/4$}
\psfrag{b}[][]{ $n/8$}
\psfrag{c}[][]{ $n/8-1$}
\psfrag{d}[][]{ $n/4+1$}
\psfrag{e}[][]{$n/4$}
\psfrag{x}[][]{$A'$}
\psfrag{y}[][]{$A''$}
\psfrag{w}[][]{$D'$}
\psfrag{z}[][]{$D''$}
\includegraphics[width=0.9\columnwidth]{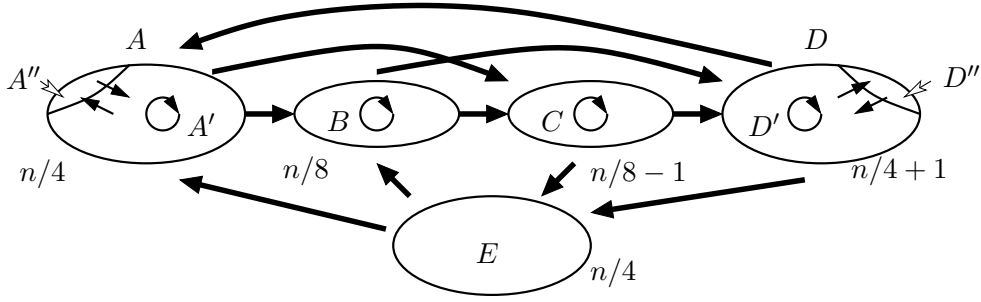} 
\caption{The oriented graph~$G$ in Proposition~\ref{orientedprop}}
\end{center}
\end{figure}
So certainly $d^+ _G (x), d^- _G (x) \geq 3n/8$ for all $x \in B \cup C \cup E$.
Furthermore, currently, $d^+ _G (a)=n/4-1$, $d^- _G (a) =n/2 +1$, $d^+ _G (d)=n/2$ and 
$d^- _G (d) =n/4-1$ for all $a \in A$ and all $d \in D$.

Partition $A$ into $A'$ and $A''$ where $|A''|=c$ and thus $|A'|=n/4-c$. Write
$A'=:\{x_1,x_2, \dots, x_{n/8-c/2},y_1,y_2, \dots, y_{n/8-c/2} \}$ and
$A''=: \{z_1,\dots, z_{2t}, w_1 , \dots , w_{2t}\}$.
Let~$A'$ induce a tournament that is as regular as possible. In particular, every vertex
in~$A'$ sends out at least $n/8-c/2-1$  edges to other vertices in~$A'$. We define the edges between~$A'$
and~$A''$ as follows: Add the edges $x_i z_j, y_iw_j$ to~$G$ for all $1 \leq i \leq n/8-c/2$
and $1\leq j \leq 2t$. Note that we can partition both $\{x_1, \dots , x_{n/8-c/2}\}$ and
$\{y_1, \dots , y_{n/8-c/2}\}$ into~$t$ sets of size $s:=n/(2c)-2$.
For each $0\leq i\leq t-1$ add all possible edges from 
$\{x_{si+1}, \dots, x_{s(i+1)}\}$ to $\{w_{2i+1}, w_{2i+2}\}$ and from
$\{y_{si+1}, \dots, y_{s(i+1)}\}$ to $\{z_{2i+1}, z_{2i+2}\}$.
If $a' \in A'$ and $a'' \in A''$ are such that the edge~$a'a''$ has not been included into~$G$ so far
then add the edge~$a''a'$ to~$G$.
Thus, $d^+ _G (a') \geq (n/4-1)+(n/8-c/2-1)+c/2+2=3n/8$ for all $a' \in A'$ and
$$d^+ _G (a'') \geq (n/4-1)+(n/8-c/2-s)=3n/8 -c/2-n/(2c)+1 \geq \alpha n$$ for all $a'' \in A''$. 

Partitioning $D$ into~$D'$ and~$D''$ (where $|D''|=c$) and defining edges inside~$D$ in a
similar fashion to those inside~$A$, we can ensure that $d^- _G (d') \geq 3n/8$ for all
$d' \in D'$ and $d^- _G (d'') \geq \alpha n$ for all $d'' \in D''$.
So indeed~$G$ has the desired degree sequences.

$E$ is an independent set, so if~$G$ contains a Hamilton cycle~$H$ then the inneighbour of
each vertex in~$E$ on~$H$ must lie in $C \cup D$ while its outneighbour lies in $A \cup B$. So~$H$
contains at least $|E|=n/4$ disjoint edges going from $A \cup B$ to $C \cup D$. However,
all such edges in~$G$ have at least one endvertex in~$B\cup C$. So there are at most
$|B|+|C|=n/4-1<|E|$ such disjoint edges in~$G$. Thus~$G$ does not contain a Hamilton
cycle (in fact, $G$ does not contain a $1$-factor).
\end{proof}

\section{The Diregularity lemma and other tools}\label{sec3}
In the proof of Theorem~\ref{approxnw} we will use the directed version of
Szemer\'edi's Regularity lemma. Before we can state it we need some more definitions. 
The {\it density} of an undirected bipartite graph $G=(A,B)$ with vertex classes~$A$ and~$B$ is defined to be 
$$d_G (A,B):=\frac{e_G(A,B)}{|A||B|}.$$ 
We will write $d(A,B)$ if this is unambiguous. Given any $\varepsilon >0$ we say that~$G$
is {\it $\varepsilon$-regular} if for all $X\subseteq A$ and $Y \subseteq B$ with $|X|>\varepsilon 
|A|$ and $|Y|> \varepsilon |B|$ we have that $|d(X,Y)-d(A,B)|<\varepsilon$.

Given disjoint vertex sets~$A$ and~$B$ in a digraph~$G$, we write $(A,B)_G$ for the
oriented bipartite subgraph of~$G$
whose vertex classes are~$A$ and~$B$ and whose edges are all the edges from~$A$ to~$B$ in~$G$.
We say $(A,B)_G$ is {\it $\varepsilon$-regular and has density~$d$} if the
underlying bipartite graph of $(A,B)_G$ is $\varepsilon$-regular and has density~$d$.
(Note that the ordering of the pair $(A,B)$ is important here.)

The Diregularity lemma is a variant of the Regularity lemma for digraphs due to Alon
and Shapira~\cite{alon}. Its proof is similar to the undirected version.
We will use the degree form of the Diregularity lemma which is derived
from the standard version in the same manner as the undirected degree form
(see e.g.~the survey~\cite{survey} for a sketch of the undirected version).

\begin{lemma}[Degree form of the Diregularity lemma]\label{dilemma}
For every $\varepsilon\in (0,1)$ and every integer~$M'$ there are integers~$M$ and~$n_0$
such that if~$G$ is a digraph on $n\ge n_0$ vertices and
$d\in[0,1]$ is any real number, then there is a partition of the vertex set of~$G$ into
$V_0,V_1,\ldots,V_k$ and a spanning subdigraph~$G'$ of~$G$ such that the following holds:
\begin{itemize}
\item $M'\le k\leq M$,
\item $|V_0|\leq \varepsilon n$,
\item $|V_1|=\dots=|V_k|=:m$,
\item $d^+_{G'}(x)>d^+_G(x)-(d+\varepsilon)n$ for all vertices $x\in V(G)$,
\item $d^-_{G'}(x)>d^-_G(x)-(d+\varepsilon)n$ for all vertices $x\in V(G)$,
\item for all $i=1,\dots,k$ the digraph $G'[V_i]$ is empty,	
\item for all $1\leq i,j\leq k$ with $i\neq j$ the pair $(V_i,V_j)_{G'}$ is $\varepsilon$-regular
and has density either~$0$ or density at least~$d$.
\end{itemize}
\end{lemma}
We call $V_1, \dots, V_k$ {\it clusters}, $V_0$ the {\it exceptional set} and the
vertices in~$V_0$ {\it exceptional vertices}. We refer to~$G'$ as the {\it pure digraph}.
The last condition of the lemma says that all pairs of clusters are $\eps$-regular in both
directions (but possibly with different densities).
The {\it reduced digraph~$R$ of~$G$ with parameters $\varepsilon$, $d$ and~$M'$} is the digraph whose 
vertices are $V_1, \dots , V_k$ and in which $V_i V_j$ is an edge precisely when $(V_i,V_j)_{G'}$
is $\varepsilon$-regular and has density at least~$d$.

Given $0<\nu\le \tau<1$, we call a digraph~$G$ a \emph{$(\nu,\tau)$-outexpander} if
$|N^+(S)|\ge |S|+\nu |G|$ for all $S\subseteq V(G)$ with $\tau |G|<|S|< (1-\tau)|G|$.
The main tool in the proof of Theorem~\ref{approxnw} is the following result from~\cite{kko}.
\begin{lemma}\label{cyclelemma}
Let $M',n_0$ be positive integers and let $\varepsilon,d,\eta,\nu,\tau$ be positive constants such that
$1/n_0\ll 1/M'\ll  \varepsilon \ll d\ll \nu\le \tau\ll \eta< 1$. Let~$G$ be an oriented graph on $n\ge n_0$
vertices such that $\delta^0(G)\ge 2\eta n$. Let~$R$ be the reduced digraph of $G$ with
parameters $\varepsilon$, $d$ and~$M'$. Suppose that there exists a spanning oriented subgraph~$R^*$
of~$R$ with $\delta^0(R^*)\ge \eta |R^*|$ which is a $(\nu,\tau)$-outexpander.
Then $G$ contains a Hamilton cycle.
\end{lemma}
Here we write $0<a_1 \ll a_2 \ll a_3 \le 1$ to mean that we can choose the constants
$a_1,a_2,a_3$ from right to left. More
precisely, there are increasing functions $f$ and $g$ such that, given
$a_3$, whenever we choose some $a_2 \leq f(a_3)$ and $a_1 \leq g(a_2)$, all
calculations needed in the proof of Lemma~\ref{cyclelemma} are valid.
 
Our next aim is to show that any digraph~$G$ as in Theorem~\ref{approxnw} is an
outexpander. In fact, we will show
that even the `robust outneighbourhood' of any set~$S\subseteq V(G)$ of reasonable size is
significantly larger than~$S$. More precisely, let $0<\nu\le  \tau<1$. Given any digraph~$G$ and
$S\subseteq V(G)$, the \emph{$\nu$-robust outneighbourhood~$RN^+_{\nu,G}(S)$ of~$S$}
is the set of all those vertices~$x$ of~$G$ which have at least $\nu |G|$ inneighbours
in~$S$. $G$ is called a \emph{robust $(\nu,\tau)$-outexpander}
if $|RN^+_{\nu,G}(S)|\ge |S|+\nu |G|$ for all
$S\subseteq V(G)$ with $\tau |G|< |S|< (1-\tau)|G|$.

\begin{lemma}\label{robustG}
Let $n_0$ be a positive integer and $\tau,\eta$ be positive constants such that
$1/n_0\ll\tau \ll \eta < 1$. Let $G$ be a digraph on $n\ge n_0$ vertices with
\begin{itemize}
\item[(i)] $d^{+} _i\geq i + \eta n$  or $d^{-} _{n-i-\eta n} \geq n-i$,
\item[(ii)] $d^{-} _i\geq i + \eta n$  or $d^{+} _{n-i-\eta n} \geq n-i$
\end{itemize}
for all $i <n/2$. Then $\delta^0(G)\ge \eta n$ and~$G$ is a robust $(\tau^2,\tau)$-outexpander.
\end{lemma}
\proof Clearly, if $d^+_1\ge 1+\eta n$ then $\delta^+(G)\ge \eta n$.
If $d^+_1< 1+\eta n$ then~(i) implies that $d^-_{n-1-\eta n}\ge n-1$.
Thus~$G$ has at least $\eta n+1$ vertices of indegree $n-1$ and so $\delta^+(G)\ge \eta n$.
It follows similarly that $\delta^-(G)\ge \eta n$.

Consider any non-empty set $S\subseteq V(G)$ with $\tau n< |S|< (1-\tau)n$
and $|S|\neq n/2+\lfloor\tau n\rfloor$. 
Let us first deal with the case when%
    \COMMENT{Have to take floors since $|S|-\tau n$ might be tiny.}
$d^+_{|S|-\lfloor \tau n\rfloor} \ge |S|-\lfloor \tau n\rfloor+\eta n\ge |S|+\eta n/2$.
Then~$S$ contains a set~$X$ of~$\lfloor \tau n\rfloor$ vertices, each having outdegree at least
$|S|+\eta n/2$. Let~$Y$ be the set of all those vertices of~$G$ that have at least
$\tau^2 n$ inneighbours in~$X$. Then $$
|X|(|S|+\eta n/2)\le |Y||X|+(n-|Y|) \tau^2 n\le |Y||X|+ \tau^2 n^2
$$
and so%
     \COMMENT{We get $|Y|\ge |S|+\eta n/2 -\frac{\tau^2 n^2}{|X|}\ge
|S|+\eta n/2 -\frac{\tau^2 n^2}{\lfloor \tau n\rfloor}$. Also, we need to have some
room for the case when $|S|= n/2+\lfloor \tau n\rfloor$, so we need the $2\tau^2 n$ instead
of $\tau^2 n$.}
$|RN^+_{\tau^2,G}(S)|\ge |Y|\ge |S|+2\tau^2 n$.

So suppose next that $d^+_{|S|-\lfloor \tau n\rfloor} < |S|-\lfloor \tau n\rfloor+\eta n$.
Since $\delta^-(G)\ge \eta n$ we may assume that $|S|\le (1-\eta+\tau^2)n<
n-1-\eta n+\lfloor \tau n\rfloor$ (otherwise $RN^+_{\tau^2,G}(S)=V(G)$ and we are done).
Thus%
    \COMMENT{ok since $n-|S|-\eta n+\lfloor \tau n\rfloor\ge 1$ by above assumption}
$$d^-_{n-|S|+ \lfloor \tau n\rfloor-\eta n}\ge n-|S|+\lfloor \tau n\rfloor\ge n-|S|+\tau^2n
$$
by~(i) and~(ii).%
    \COMMENT{(Note that this also holds if $|S|-\lfloor \tau n\rfloor>n/2$:
Take $i:=n-|S|+ \lfloor \tau n\rfloor-\eta n$. Then $i<n/2$ and
$n-i-\eta n=|S|-\lfloor \tau n\rfloor$.}
(Here we use that $|S|\neq n/2+\lfloor\tau n\rfloor$.)

So~$G$ contains at least $|S|-\lfloor \tau n\rfloor+\eta n\ge |S|+\eta n/2$ vertices~$x$
of indegree at least $n-|S|+\tau^2 n$. If $|RN^+_{\tau^2,G}(S)|< |S|+2\tau ^2 n$
then $V(G)\setminus RN^+_{\tau^2,G}(S)$ contains such a vertex~$x$.
But then~$x$ has at least~$\tau^2 n$ neighbours in~$S$, i.e.~$x\in RN^+_{\tau^2,G}(S)$, a contradiction.  

If $|S|= n/2+\lfloor \tau n\rfloor$ then considering the outneighbourhood of a subset of~$S$ of
size $|S|-1$ shows that $|RN^+_{\tau^2,G}(S)|\ge |S|-1+2\tau^2 n\ge |S|+\tau^2n$.
\endproof

The next result implies that the property of a digraph~$G$ being a robust outexpander is `inherited'
by the reduced digraph of~$G$. For this (and for Lemma~\ref{orientexp}) we need that~$G$
is a robust outexpander, rather than just an outexpander.

\begin{lemma}\label{robustR}
Let $M',n_0$ be positive integers and let $\varepsilon,d,\eta,\nu,\tau$ be positive constants such that
$1/n_0\ll  \varepsilon \ll d\ll \nu,\tau,\eta < 1$ and such that $M'\ll n_0$. Let~$G$ be a
digraph on $n\ge n_0$ vertices with $\delta ^0 (G) \ge \eta n$ and such that~$G$
is a robust $(\nu,\tau)$-outexpander. Let~$R$ be the
reduced digraph of $G$ with parameters $\varepsilon$, $d$ and~$M'$. Then $\delta^0(R)\ge \eta |R|/2$
and~$R$ is a robust $(\nu/2,2\tau)$-outexpander.
\end{lemma}
\proof
Let~$G'$ denote the pure digraph, $k:=|R|$, let~$V_1,\dots,V_k$ be the clusters of~$G$
(i.e.~the vertices of~$R$) and~$V_0$ the exceptional set. Let $m:=|V_1|=\dots=|V_k|$.
Then $$\delta^0(R)\ge (\delta^0(G')-|V_0|)/m\ge (\delta^0(G)-(d+2\eps)n)/m\ge \eta k/2.$$

Consider any $S\subseteq V(R)$ with $2\tau k\le |S|\le (1-2\tau)k$. 
Let~$S'$ be the union of all the clusters belonging to~$S$.
Then $\tau n \le |S'|\le (1-2\tau)n$. Since $|N^-_{G'}(x)\cap S'|\ge |N^-_{G}(x)\cap S'|-(d+\eps)n
\ge \nu n/2$ for every $x\in RN^+_{\nu,G}(S')$ this implies that
$$
|RN^+_{\nu/2,G'}(S')|\ge |RN^+_{\nu,G}(S')|\ge |S'|+\nu n\ge |S|m+\nu mk.
$$
However, in~$G'$ every vertex $x\in RN^+_{\nu/2,G'}(S')\backslash V_0$ receives edges from vertices in at least
$|N^-_{G'}(x)\cap S'|/m\ge (\nu n/2)/m \ge \nu k/2$ clusters~$V_i\in S$. Thus 
by the final property of the partition in Lemma~\ref{dilemma} the cluster~$V_j$ containing~$x$
is an outneighbour of each such~$V_i$ (in~$R$). Hence $V_j\in RN^+_{\nu/2,R}(S)$. This
in turn implies that
$$|RN^+_{\nu/2,R}(S)|\ge (|RN^+_{\nu/2,G'}(S')|-|V_0|)/m\ge |S|+\nu k/2,
$$
as required.
\endproof

The strategy of the proof of Theorem~\ref{approxnw} is as follows. By Lemma~\ref{robustG}
our given digraph~$G$ is a robust outexpander and by Lemma~\ref{robustR} this also
holds for the reduced digraph~$R$ of~$G$. The next result 
gives us a spanning oriented subgraph~$R^*$ of~$R$ which is still an
outexpander. The somewhat technical property concerning the subdigraph~$H\subseteq R$
in Lemma~\ref{orientexp} will be used to guarantee an oriented subgraph~$G^*$ of $G$
which has linear minimum semidegree and such that~$R^*$ is a reduced digraph of~$G^*$.
($G^*$ will be obtained from the spanning subgraph of the pure digraph~$G'$
which corresponds to~$R^*$ by modifying the neighbourhoods of a small number of vertices.)
Finally, we will apply Lemma~\ref{cyclelemma} with~$R^*$ playing the role of both~$R$ and~$R^*$
and~$G^*$ playing the role of~$G$ to find a Hamilton cycle in~$G^*$ and thus in~$G$.

\begin{lemma}\label{orientexp}
Given positive constants $\nu\le \tau\le \eta$, there exists a positive integer~$n_0$
such that the following holds. Let~$R$ be a digraph on~$n\ge n_0$ vertices which is a
robust $(\nu,\tau)$-outexpander. Let~$H$ be a spanning subdigraph of~$R$ with
$\delta^0(H)\ge \eta n$. Then $R$ has
a spanning oriented subgraph~$R^*$ which is a robust $(\nu/12,\tau)$-outexpander and such
that $\delta^0(R^*\cap H)\ge \eta n/4$.
\end{lemma}
\proof
Consider a random spanning oriented subgraph~$R^*$ of~$R$ obtained by deleting one
of the edges $xy,yx$ (each with probability 1/2) for every pair $x,y\in V(R)$
for which $xy,yx\in E(R)$, independently from all other such pairs.
Given a vertex~$x$ of~$R$, we write $N^{\pm}_R(x)$ for the set of all those vertices of~$R$
which are both out- and inneighbours of~$x$ and define $N^{\pm}_H(x)$ similarly.
Let $H^*:=H\cap R^*$. Clearly, $d^+_{H^*}(x), d^-_{H^*}(x)\ge \eta n/4$
if $|N^{\pm}_H(x)|\le 3\eta n/4$. So suppose that $|N^{\pm}_H(x)|\ge 3\eta n/4$. Let
$X:=|N^{\pm}_H(x)\cap N^+_{H^*}(x)|$. Then $\ex X \ge 3\eta n/8$ and so a standard Chernoff
estimate (see e.g.~\cite[Cor.~A.14]{ProbMeth}) implies that
$$
\prob(d^+_{H^*}(x)<\eta n/4)\le \prob(X < \eta n/4 ) \le 
\prob(X< 2\ex X/3) <2\eul^{-c\ex X}\le 2\eul^{-3c\eta n/8},
$$
where~$c$ is an absolute constant (i.e.~it does not depend on $\nu$, $\tau$ or $\eta$).
Similarly it follows that $\prob(d^-_{H^*}(x)<\eta n/4)\le 2\eul^{-3c\eta n/8}$.

Consider any set $S\subseteq V(R^*)=V(R)$.
Let $ERN^+_{ \nu/3 , R}(S):=RN^+_{ \nu/3 , R}(S)\setminus S$ and define
$ERN^+_{\nu/12,R^*}(S)$ similarly. We say that~$S$ is \emph{good} if all
but at most~$\nu n/6$ vertices in $ERN^+_{ \nu/3 , R}(S)$ are contained
in $ERN^+_{\nu/12,R^*}(S)$. Our next aim is to show that
\begin{equation}\label{eq:good}
\prob(S \text{ is not good})\le \eul^{- n}.
\end{equation}
To prove~(\ref{eq:good}), write $ERN^{\pm}_R(S)$ for the set of all those
vertices $x\in ERN^+_{ \nu/3 , R}(S)$ for which $|N^{\pm}_R(x)\cap S|\ge \nu n/4$.
Note that every vertex in $ERN^+_{ \nu/3 , R}(S)\setminus ERN^{\pm}_{R}(S)$
will automatically lie in $ERN^+_{\nu/12,R^*}(S)$.
We say that a vertex $x\in ERN^{\pm}_R(S)$ \emph{fails}
if $x\notin ERN^+_{\nu/12,R^*}(S)$. 
The expected size of $N^-_{R^*}(x)\cap N^{\pm}_R(x)\cap S$ is at least $\nu n/8$.
So as before, a Chernoff estimate gives
$$
\prob(x \text{ fails})\le\prob(|N^-_{R^*}(x)\cap N^{\pm}_R(x)\cap S|<\nu n/12)\le 2\eul^{-c\nu n/8}=:p.
$$
Let~$Y$ be the number of all those vertices $x\in ERN^{\pm}_R(S)$ which fail.
Then $\ex Y \le p|ERN^{\pm}_R(S)| \le pn$.
Note that the failure of distinct vertices is independent (which is the reason we are only
considering vertices in the external neighbourhood of $S$).
So we can apply the following Chernoff estimate
(see e.g.~\cite[Theorem~A.12]{ProbMeth}): If $C \ge \eul^2$ we have
$$
\prob(Y\ge C\ex Y)\le \eul^{(C -C\ln C)\ex Y}\le \eul^{-C(\ln C) \ex Y/2}.
$$
Setting $C:= \nu n/(6\ex Y)\ge \nu/(6p)$ this gives
\begin{align*}
\prob (S \mbox{ is not good}) & =
\prob(Y> \nu n/6)=\prob(Y>  C\ex Y) \le \eul ^{-C (\ln C) \ex Y/2} =
\eul ^{-\nu n (\ln C) /12} \\ 
& \le \eul^{-n}.
\end{align*}
(The last inequality follows since $p \ll \nu$ if~$n$ is sufficiently large.)
This completes the proof of~(\ref{eq:good}).

Since $4n\eul^{-3c\eta n/8}+2^n \eul^{-n}<1$ (if $n$ is sufficiently large)
this implies that there is an outcome for~$R^*$ such that $\delta^0(R^*\cap H)\ge \eta n/4$
and such that every set $S\subseteq V(R)$ is good.
We will now show that the latter property implies that such an~$R^*$ is a
robust $(\nu/12,\tau)$-outexpander.
So consider any set~$S \subseteq V(R)$ with $\tau n < |S| <(1-\tau) n$. 
Let $EN:=ERN^+_{\nu,R}(S)$ and
$N:=RN^+_{\nu,R}(S)\cap S$. So $EN\cup N= RN^+_{\nu,R}(S)$. Since~$S$ is good and
$EN\subseteq ERN^+_{\nu/3,R}(S)$ all but at most $\nu n/6$ vertices in~$EN$ are contained
in $ERN^+_{\nu/12,R^*}(S)\subseteq RN^+_{\nu/12,R^*}(S)$.

Now consider any partition of~$S$ into~$S_1$ and~$S_2$ such that every vertex $x\in N$
satisfies $|N^-_R(x)\cap S_i|\ge \nu n/3$ for $i=1,2$. (The existence of such a partition
follows by considering a random partition.) Then $S_1\cap N \subseteq ERN^+_{\nu/3,R}(S_2)$.
But since~$S_2$ is good this implies that all but at most $\nu n/6$ vertices in~$S_1\cap N$
are contained in $ERN^+_{\nu/12,R^*}(S_2)\subseteq RN^+_{\nu/12,R^*}(S)$.
Similarly, since~$S_1$ is good, all but at most $\nu n/6$ vertices in~$S_2\cap N$
are contained in $ERN^+_{\nu/12,R^*}(S_1)\subseteq RN^+_{\nu/12,R^*}(S)$.
Altogether this shows that
$$|RN^+_{\nu/12,R^*}(S)|\ge |EN\cup (S_1\cap N)\cup (S_2\cap N)|-\frac{3\nu n}{6}
=|RN^+_{\nu,R}(S)|-\frac{\nu n}{2}\ge |S|+\frac{\nu n}{2},
$$
as required.
\endproof

\section{Proof of Theorem~\ref{approxnw}}\label{sec4}

As indicated in Section~\ref{sec1}, instead of proving Theorem~\ref{approxnw} directly,
we will prove the following stronger result. It immediately implies Theorem~\ref{approxnw}
since by Lemma~\ref{robustG} any digraph~$G$ as in~Theorem~\ref{approxnw} is a robust outexpander
and satisfies $\delta^0(G)\ge \eta n$.

\begin{thm}\label{expanderthm}
Let $n_0$ be a positive integer and $\nu,\tau,\eta$ be
positive constants such that $1/n_0\ll\nu\le \tau\ll\eta<1$. Let~$G$ be a digraph on~$n\ge n_0$ vertices with
$\delta^0(G)\ge \eta n$ which is a robust $(\nu,\tau)$-outexpander. Then~$G$ contains a Hamilton cycle.
\end{thm}
\proof Pick a positive integer $M'$ and additional constants~$\eps,d$ such that
$1/n_0\ll 1/M'\ll \eps\ll d\ll \nu$. 
Apply the Regularity lemma (Lemma~\ref{dilemma}) with parameters~$\eps$, $d$ and~$M'$ to~$G$
to obtain clusters $V_1,\dots,V_k$, an exceptional set~$V_0$ and a pure digraph~$G'$.
Then $\delta^0(G')\ge (\eta -(d+\eps))n$ by Lemma~\ref{dilemma}.
Let~$R$ be the reduced digraph of $G$ with parameters $\eps$, $d$ and~$M'$.
Lemma~\ref{robustR} implies that $\delta^0(R)\ge \eta k/2$ and that~$R$ is a robust
$(\nu/2,2\tau)$-outexpander. 

Let~$H$ be the spanning subdigraph of~$R$ in which~$V_iV_j$ is an edge if $V_iV_j\in E(R)$
and the density $d_{G'}(V_i,V_j)$ of the oriented subgraph $(V_i,V_j)_{G'}$ of~$G'$
is at least~$\eta/4$. We will now give a
lower bound on~$\delta^+(H)$. So consider any cluster~$V_i$ and let $m:=|V_i|$.
Writing $e_{G'}(V_i,V(G)\setminus V_0)$ for the number of all edges from~$V_i$
to~$V(G)\setminus V_0$ in~$G'$, we have 
$$\sum_{V_j\in N^+_R(V_i)} d_{G'}(V_i,V_j) m^2= e_{G'}(V_i, V(G)\setminus V_0)\ge \delta^0(G')m-|V_0|m
\ge (\eta-2d)nm.
$$
It is easy to see%
     \COMMENT{Indeed, if not then $\sum_{V_j\in N^+_R(V_i)} d_{G'}(V_i,V_j)\le
1\cdot \eta k/4+(\eta/4)\cdot k= \eta k/2$, a contradiction.}
that this implies that there are at least~$\eta k/4$ outneighbours~$V_j$ of~$V_i$ in~$R$ such
that $d_{G'}(V_i,V_j)\ge \eta/4$. But each such~$V_j$ is an outneighbour of~$V_i$ in~$H$
and so $\delta^+(H)\ge \eta k/4$. It follows similarly that $\delta^-(H) \ge \eta k/4$.
We now apply Lemma~\ref{orientexp}  to find a
spanning oriented subgraph~$R^*$ of~$R$ which is a (robust) $(\nu/24, 2\tau)$-outexpander
and such that $\delta^0(R^*\cap H)\ge \eta k/16$. Let $H^*:=H\cap R^*$.

Our next aim is to modify the pure digraph~$G'$ into a spanning oriented subgraph of~$G$ having minimum
semi-degree at least $\eta^2 n/100$. Let~$G^*$ be the spanning subgraph of~$G'$
which corresponds to~$R^*$. So~$G^*$ is obtained from~$G'$ by deleting all those edges~$xy$
that join some cluster~$V_i$ to some cluster~$V_j$ with $V_iV_j\in E(R)\setminus E(R^*)$.
Note that $G^*-V_0$ is an oriented graph. However, some vertices of $G^*-V_0$ may have
small degrees. We will show that there are only a few such vertices and we will add them
to~$V_0$ in order to achieve that the out- and indegrees of all the vertices outside~$V_0$ are
large. So consider any
cluster~$V_i$. For any cluster $V_j\in N^+_{H^*}(V_i)$ at most $\eps m$ vertices
in~$V_i$ have less than $(d_{G'}(V_i,V_j)-\eps)m\ge \eta m/5$ outneighbours in~$V_j$ (in the digraph~$G'$).
Call all these vertices of~$V_i$ \emph{useless for~$V_j$}. Thus on average any vertex
of~$V_i$ is useless for at most $\eps |N^+_{H^*}(V_i)|$ clusters $V_j\in N^+_{H^*}(V_i)$.
This implies that at most $\sqrt{\eps} m$ vertices in~$V_i$ are useless for
more than $\sqrt{\eps} |N^+_{H^*}(V_i)|$ clusters $V_j\in N^+_{H^*}(V_i)$. Let~$U^+_i\subseteq V_i$
be a set of size $\sqrt{\eps} m$ which consists of all these vertices and some extra vertices
from~$V_i$ if necessary. Similarly, we can choose a set~$U^-_i\subseteq V_i \setminus U_i^+$ 
of size $\sqrt{\eps} m$ such that for every
vertex $x\in V_i\setminus U^-_i$ there are at most $\sqrt{\eps} |N^-_{H^*}(V_i)|$ clusters
$V_j\in N^-_{H^*}(V_i)$ such that~$x$ has less than $\eta m/5$ inneighbours in~$V_j$.
For each $i=1,\dots,k$ remove all
the vertices in~$U^+_i\cup U^-_i$ and add them to~$V_0$. We still denote the subclusters
obtained in this way by $V_1,\dots,V_k$ and the exceptional set by~$V_0$. Thus we now have
that $|V_ 0|\le 3\sqrt{\eps} n$. Moreover, 
$$\delta^0(G^*-V_0)\ge \frac{\eta m}{5} (1-\sqrt{\eps})\delta^0(H^*)-|V_0|
\ge\frac{\eta m}{5}\frac{\eta k}{17} -3 \sqrt{ \eps} n \ge \frac{\eta^2 n}{100}.
$$
We now modify~$G^*$ by altering the neighbours of the exceptional vertices:
For every $x\in V_0$ we select a set of $\eta n/2$ outneighbours of~$x$ in~$G$
and a set of $\eta n/2$ inneighbours such that these two sets are disjoint and add the
edges between $x$ and the selected neighbours to~$G^*$. We still denote
the oriented graph thus obtained from~$G^*$ by~$G^*$. Then $\delta^0(G^*)\ge \eta^2 n/100$.
Since the partition $V_0,V_1,\dots,V_k$ of $V(G^*)$ is as described in the Regularity lemma
(Lemma~\ref{dilemma}) with parameters $3\sqrt{\eps}$, $d-\eps$ and $M'$
(where $G^*$ plays the role of $G'$ and $G$) we can say  
that~$R^*$ is a reduced digraph of~$G^*$ with these parameters.%
\COMMENT{Need $3\sqrt{\eps}$ instead of $\eps/2$ (say) to satisfy the condition on $V_0$}
Thus we may apply Lemma~\ref{cyclelemma} with~$R^*$ playing the role of both~$R$ and~$R^*$
and~$G^*$ playing the role of~$G$ to find a Hamilton cycle in~$G^*$ and thus in~$G$.
\endproof
\section{Hamilton cycles in regular tournaments}\label{torn}
In this section we prove Conjecture~\ref{thomconj} for sufficiently large 
regular tournaments. The following observation of Keevash and Sudakov~\cite{keevs} will be useful for this.

\begin{prop}\label{kands} Let $0<c<10^{-4}$ and let $G$ be an oriented graph on $n$ vertices such that $\delta ^0 (G) \geq (1/2-c)n$.
Then for any (not necessarily disjoint) $S,T \subseteq V(G)$ of size at least $(1/2-c)n$ there are at least $n^2/60$
directed edges from $S$ to $T$.
\end{prop}

We now show that Theorem~\ref{expanderthm} implies Conjecture~\ref{thomconj} for sufficiently large regular tournaments. 
\begin{cor}\label{thomconjlarge} There exists an integer $n_0$ such that the following holds. Given any regular tournament $G$ on $n \geq n_0$
vertices and a set $A$ of less than $(n-1)/2$ edges of $G$, then $G-A$ contains a Hamilton cycle.
\end{cor}
\proof
Let $0<\nu \ll \tau \ll \eta \ll 1$. It is not difficult to show that $G$ is a robust $(\nu, \tau)$-outexpander. Indeed, if 
$S \subseteq V(G)$ and $(1/2+\tau)n<|S|<
(1- \tau)n$ then $RN^+ _{\nu,G} (S)=V(G)$. If $\tau n <|S|<(1/2-\tau)n $ then it is easy to see that
$|RN^+ _{\nu,G} (S)|\geq (1-\tau)n/2 \ge |S|+ \nu n$.%
    \COMMENT{Since $S$ sends out $(n-1)|S|/2$ edges we have that
$(n-1)|S|/2 \leq |S||RN ^+ _{\nu, G} (S)|+ \nu n^2$. Thus
$|RN ^+ _{\nu, G} (S)|\geq (n-1)/2-\nu n^2/|S|\ge (n-1)/2-\nu n/\tau\geq n/2 - \tau n/2 \geq |S|+ \nu n$.}
So consider the case when $(1/2 -\tau)n \leq |S| \leq (1/2+\tau)n$. Suppose $|RN ^+ _{\nu, G} (S)|< |S|+ \nu n \leq (1/2+2\tau)n$. Then by 
Proposition~\ref{kands} there are at least $n^2/60$ directed edges from $S$ to $V(G) \backslash RN^+ _{\nu ,G} (S)$. By definition each vertex
$x \in V(G) \backslash RN^+ _{\nu ,G} (S)$ has less than $\nu n$ inneighbours in $S$, a contradiction. So
$|RN^+ _{\nu,G} (S)|\geq |S|+ \nu n$ as desired.

Since $|A|< (n-1)/2$ and $n$ is sufficiently large, $G-A$ must be a robust $(\nu/2, \tau)$-outexpander. Thus if $\delta ^0 (G-A) 
\geq \eta n$ then by Theorem~\ref{expanderthm}, $G-A$ contains a Hamilton cycle. 

If $\delta ^0 (G-A)< \eta n$ then there exists precisely one vertex $x \in V(G-A)$ such that either
$d^+ _{G-A} (x) < \eta n$ or $d^- _{G-A} (x) < \eta n$.
Without loss of generality we may assume that $d^+ _{G-A} (x) < \eta n$. Note that $d^+ _{G-A} (x)\geq 1$ and let $y \in N^+ _{G-A} (x)$. 
Let $G'$ be the digraph obtained from $G-A$ by removing $x$ and $y$ from $G-A$ and adding a new vertex $z$
so that $N^+ _{G'} (z):=N^+ _{G-A} (y)$ and 
$N^- _{G'} (z):=N^- _{G-A} (x)$. So $\delta ^0 (G') \geq \eta n -2\geq \eta n/2$ and $G'$ is a robust $(\nu/3, 2 \tau)$-outexpander. 
Thus by Theorem~\ref{expanderthm} $G'$ contains a Hamilton cycle which corresponds to one in~$G$.
\endproof

\medskip 

{\footnotesize \obeylines \parindent=0pt

Daniela K\"{u}hn, Deryk Osthus \& Andrew Treglown
School of Mathematics
University of Birmingham
Edgbaston
Birmingham
B15 2TT
UK
}

{\footnotesize \parindent=0pt

\it{E-mail addresses}:
\tt{\{kuehn,osthus,treglowa\}@maths.bham.ac.uk}}
\end{document}